\documentclass[10pt,reqno]{amsart}
\usepackage{mathrsfs}
\usepackage{amsgen}
\usepackage{amscd}

\usepackage{amsmath}
\usepackage{latexsym}
\usepackage{amsfonts}
\usepackage{amssymb}
\usepackage{amsthm}
\usepackage{graphicx,epsfig}
\usepackage{color}
\usepackage{amsthm,amssymb}
\usepackage{graphicx}
\usepackage{enumerate}
\usepackage{amssymb,amsmath,graphicx,amsfonts,euscript}
\usepackage{color}
\usepackage{mathrsfs}
\usepackage{empheq}
\usepackage{cases}
\usepackage{cite}
\usepackage{ulem}
\usepackage{cancel}
\usepackage[margin=3cm, a4paper]{geometry}
\DeclareGraphicsExtensions{.eps}
\usepackage{amssymb,amsmath,graphicx,amsfonts,euscript,mathrsfs}
\usepackage{hyperref}

\newtheorem{theorem}{Theorem}[section]
\newtheorem{lemma}[theorem]{Lemma}

\newtheorem{corollary}[theorem]{Corollary}
\theoremstyle{definition}

\theoremstyle{remark}
\newtheorem{remark}[theorem]{Remark}

\def\R{\mathbb{R}}
\def\Z{\mathbb{Z}}
\def\T{\mathbb{T}}
\def\C{\mathbb{C}}

\newcommand{\fe}{\mathrm{e}}

\newcommand{\bR}{{\mathbb R}}

\newcommand{\bT}{{\mathbb T}}

\numberwithin{equation}{section}

\begin{document}

\title[Fourier integrator for NLS equation]{A first-order Fourier integrator for the nonlinear Schr\"odinger equation on $\T$ without loss of regularity}

\author[Y. Wu]{Yifei Wu}
\address{\hspace*{-12pt}Y.~Wu: Center for Applied Mathematics, Tianjin University, 300072, Tianjin, China}
\email{yerfmath@gmail.com}

\author[F. Yao]{Fangyan Yao}
\address{\hspace*{-12pt}F.~Yao: School of Mathematical Sciences,
South China University of Technology,
 Guangzhou, Guangdong 510640, P. R. China}
\email{yfy1357@126.com}




\maketitle

\begin{abstract}\noindent
In this paper, we  propose a first-order  Fourier integrator for solving the cubic nonlinear Schr\"odinger equation in one dimension. The scheme is explicit and  can be implemented using the fast Fourier transform. By a rigorous analysis, we prove that the new scheme provides the first order accuracy in $H^\gamma$ for any initial data belonging to $H^\gamma$, for any $\gamma >\frac32$. That is, up to some fixed time $T$,  there exists some constant $C=C(\|u\|_{L^\infty([0,T]; H^{\gamma})})>0$, such that
$$
\|u^n-u(t_n)\|_{H^\gamma(\T)}\le  C \tau,
$$
 where $u^n$ denotes the numerical solution at $t_n=n\tau$. Moreover, the mass of the numerical  solution $M(u^n)$ verifies
$$
\left|M(u^n)-M(u_0)\right|\le  C\tau^5.
$$
In particular, our scheme dose not cost any additional derivative for the first-order convergence and the numerical solution obeys the almost mass conservation law. Furthermore, if $u_0\in H^1(\T)$, we rigorously prove that
$$
\|u^n-u(t_n)\|_{H^1(\T)}\le  C\tau^{\frac12-},
$$
where $C= C(\|u_0\|_{H^1(\T)})>0$.
\end{abstract}


\section{Introduction}\label{sec:introduction}

In this paper, we are
concerned with the numerical integration of the  cubic nonlinear Schr\"odinger equation  (NLS) on a torus:
\begin{equation}\label{model}
 \left\{\begin{aligned}
& i\partial_tu(t,x)+\partial_{xx} u(t,x)
 +\lambda|u(t,x)|^2u(t,x)=0,
 \quad t>0,\ x\in\T,\\
 &u(0,x)=u_0(x),\quad x\in\bT,
 \end{aligned}\right.
\end{equation}
where $\bT=(0,2\pi)$, $\lambda=\pm1$, $u=u(t,x):\bR^{+}\times\bT\to\C$ is the unknown
and $u_0\in H^\gamma(\bT)$ with some $\gamma\ge 0$ is a given initial data.
Here we only consider the case $\lambda=-1$, and the case $\lambda=1$ is similar.

It is known that the local well-posedness of \eqref{model} has been established in $H^{\gamma}$
for $\gamma\geq0$ in  one dimension space, referring to \cite{Bo}.  Moreover, for $L^2$ solution of \eqref{model}, we have the following mass conservation law:
\begin{align}
M\big(u(t)\big)&=\frac1{2\pi}\int_\T |u(t,x)|^2\,d x =M(u_0). \label{mass}
\end{align}
Then the global well-posedness in $L^2$ is followed directly by the mass conservation law and the local theory.

There has been substantial research undertaken in numerical analysis of \eqref{model}.
In order to do numerical discretizations in space and time,
many methods have been proposed and extensively studied
by assuming that the exact solution is smooth enough,
for example  in finite difference methods, operator splitting, spectral methods,
discontinuous Galerkin methods and exponential integrators.
 Much of   the literature   focusses on these classical numerical schemes to
establish the convergence results of  the solution,
 referring to \cite{hairer,Hochbruck,holden,Splitting}.
  For the nonlinear Schr\"{o}dinger equation,
 we can further refer to   \cite{besse,cano,celledoni,cohen,djuardin,faou,ignat,jahnke,Lubich,thalhammer}.
 For the Korteweg-de Vries equation,  refer to
 \cite{courtes,hofmanova,holden2011,holden2012,gubinelli,ostermann-su} for recent works.

For splitting methods, one of the most  popular splitting is Strang splitting,
which can  speed up calculation for problems involving operators on very different time scales.
In particular, splitting methods are especially effective if the equation splits into two equations
 which can be directly integrated such as the Korteweg-de Vries equation and
 the nonlinear Schr\"{o}dinger equation.
For example, for the nonlinear Schr\"{o}dinger equation, the first-order and the second-order convergence
in $H^\gamma$ was achieved for the initial data in $H^{\gamma+2}$ and $H^{\gamma+4}$ respectively, see \cite{Lubich}.

For exponential integrators,
to the best of our knowledge, which were earlier considered by Hersch in \cite{hersch}, Certaine in \cite{certaine} and Pope in \cite{pope}.
Then Hochbruck, Lubich, and Selhofer  \cite{hochbruck-l}
 put forward  the term ``exponential integrator",
  which created a powerful push of the exponential integrator.
  Furthermore,  this work  was regarded as  the first actual
implementation of the exponential integrator.
  Recently, exponential integrators have become an active area of research and
  have a good development,
  more on the rich history and research results of exponential integrators can be found
  in the literature \cite{Hochbruck} by Hochbruck and Ostermann.
   Originally developed for solving stiff differential equations by Hochbruck and Ostermann in
   \cite{hochbruck2005a,hochbruck2005b},
   the methods have been used to solve partial differential equations including hyperbolic,
   as well as parabolic problems such as the heat equation.

As mentioned above,  the schemes above were constructed under the assumption that
the exact solution is smooth enough.
However, in practice the initial data may not be ideally smooth due to multiple reasons
such as measurements or noise.
Rough data may appear naturally in some applications:
initial data corrupted with noise (as in nonlinear optics applications).
When the solution of the equation is not sufficiently smooth in space,
the convergence of a certain order only holds under sufficient additional regularity assumptions
on the solution.
It can be regarded as the following  error structure of the scheme
$$
\tau^\nu(-\Delta)^\lambda u(t), \qquad \nu,\ \lambda\geq0,
$$
where  $\tau$ denotes the time step size.
The structure explains that there are $2\lambda$ derivatives loss to reach the $\nu$-order convergence.
For example, for the nonlinear Schr\"{o}dinger equation, the  error structure of
the scheme in \cite{Lubich} is
$$
\tau^\nu(-\Delta)^\nu u(t), \qquad \nu\geq0.
$$
It implies that  in order to obtain $\nu$-order convergence, $2\nu$ derivatives loss is needed.
Then the essential work is to design a numerical scheme
such that $\nu$-order convergence is achieved with $\lambda$ as small as possible.
 To bring down the regularity requirements,
more recent attention has focused on  so-called low-regularity integrators (LRIs)
that based on the exponential integrators.
Unlike the classical numerical schemes,
this method can break the natural order barrier to reach the optimal convergence rate.
Of course, it will encounter many difficulties,
the main difficulty is the design of LRIs, which needs the scheme being defined point-wise
 in the physical space while requiring  as few additional derivatives as possible.
 Moreover, this point-wise evaluation requires O(NlogN) operations, in general.
 The low-regularity integrators have already been considered for some important models
 such as the  nonlinear Schr\"{o}dinger equation (NLS),  the Korteweg-de Vries equation.

For the Korteweg-de Vries equation,
Hofmanov\'a and Schratz \cite{hofmanova} introduced an exponential integrator
to obtain first-order convergence in $H^1$
 with initial data in $H^3$.
 Later, Wu and Zhao \cite{WuZhao-1}
   established the second-order convergence result
   in $H^\gamma$ for initial data in $H^{\gamma+4}$,
 which proved rigorously in theory the validity of the scheme  that was proposed in  \cite{hofmanova}.
Very recently, Wu and Zhao \cite{wu} proposed the Embedded exponential-type low-regularity integrators
and proved the first-order and second-order convergence
in $H^\gamma$ under $H^{\gamma+1}$-data and $H^{\gamma+3}$-data respectively.

For the nonlinear Schr\"{o}dinger equations, Ostermann and Schratz \cite{lownls}
introduced  a new exponential-type numerical scheme,
and the first order convergence was achieved under the requirement of only one additional derivative.
That is
$$
\|u^n-u(t_n)\|_{H^\gamma(\T^d)}\lesssim \tau,
$$
up to some fixed time for the initial data $u_0\in H^{\gamma+1}(\T^d)$, $\gamma>\frac{d}{2}$,
where $u^n$ denotes the numerical solution at $t_n=n\tau$.
More precisely,
 the error behavior of  the  numerical scheme is dominated by
$$
\tau (-\Delta)^\frac12 u(t),
$$
which breaks the ``natural order barrier" of $\tau^\frac12$ for $(-\Delta)^\frac12$-loss.
This presents a lower regularity assumptions on the data compared to the splitting or
exponential integrator schemes.
Later, for the second-order convergence, Kn\"{o}ller, Ostermann and Schratz \cite{lownls2}
gave a new type of integrator
and the scheme requires two additional derivatives of the solution in one space dimension
 and three derivatives in higher space dimensions.
Whereafter, Ostermannn, Rousset and Schratz \cite{ostermann,ostermann2020}
considered $H^s$, $0<s\leq1$ solutions in $L^2$
 with order  $\nu<1$ in dimensions $d\leq 3$.
For the quadratic nonlinear
Schr\"{o}dinger equation and the nonlinear Dirac equation,
the first-order convergence in $H^\gamma$ without any loss of derivatives
in one space dimension, see \cite{lownls}.

 In this paper, we are aiming to get the first-order convergence of \eqref{model}
without any loss of derivatives
by introducing a new type low-regularity exponential integrator.
That is, we obtain the following result
$$
\|u^n-u(t_n)\|_{H^\gamma(\T)}\lesssim \tau,
$$
up to some fixed time for the initial data $u_0\in H^{\gamma}(\T)$, $\gamma>\frac32$,
where $u^n$ denotes the numerical solution at $t_n=n\tau$.

Now we explain our argument briefly. Our designation is based on the Phase-Space analysis of the nonlinear dynamics.
In particular, we focus our attention on the following time integral,
\begin{align}
\int_0^\tau \fe^{is
 ( k ^2+ k _1^2- k _2^2- k _3^2)}\,ds,
 \quad
 \mbox{ with } \quad k=k_1+k_2+k_3.
\end{align}
Using the following formula of the phase function,
$$
 k ^2+ k _1^2- k _2^2- k _3^2=2(k_1+k_2)(k_1+k_3),
$$
we write that for any $k\ne 0$,
\begin{align*}
 \fe^{is
 ( k ^2+ k _1^2- k _2^2- k _3^2)}
 =&\sum\limits_{j=2,3}\frac{k_1+k_j}{k} \fe^{is
 ( k ^2+ k _1^2- k _2^2- k _3^2)}
  -\frac{k_1}{k} \fe^{is
 ( k ^2+ k _1^2- k _2^2- k _3^2)}.
\end{align*}
Then the first  term is integrable. Indeed,  by direct calculation  we have that
\begin{align*}
\sum\limits_{j=2,3}\int_0^\tau \frac{k_1+k_j}{k} \fe^{is
 ( k ^2+ k _1^2- k _2^2- k _3^2)}\,ds
 =\sum\limits_{j=2,3} \frac{1}{2ik(k_1+k_j)} \left(\fe^{i\tau
 ( k ^2+ k _1^2- k _2^2- k _3^2)}-1\right).
\end{align*}
Unfortunately,  the second term can not be integrated
in the physical space exactly. To overcome this difficulty,  we use another formula of the phase function,
$$
 k ^2+ k _1^2- k _2^2- k _3^2=2kk_1+2k_2k_3.
$$
Therefore we have the formula
$$
\fe^{is
 ( k ^2+ k _1^2- k _2^2- k _3^2)}
   =  \fe^{2iskk_1}+O(s|k_2||k_3|).
$$
This yields that  for any $k\ne 0$,
\begin{align*}
\int_0^\tau \frac{k_1}{k} \fe^{is
 ( k ^2+ k _1^2- k _2^2- k _3^2)}\,ds
= \frac{1}{2ik^2} \left(\fe^{2i\tau
 kk_1}-1\right)+\tau^2O( |k|^{-1}|k_1| |k_2| |k_3|).
\end{align*}
Based on the rigorous analysis, we construct the following numerical solution of \eqref{model} as
\begin{align}\label{numerical}
u^{n+1}=\Psi(u^n), \quad n=0,1\ldots,\frac{T}{\tau}-1,
\end{align}
with $u^0=u_0$, where
\begin{align}\label{psi}
\Psi(u^n)=&\fe^{i\tau\left(-2M_0-2P_0\partial_x^{-1}+\partial_x^2\right)}u^n
               -i\tau \Pi_0\left(|u^n|^2u^n\right)+2i\tau M_0 \Pi_0(u^n)\notag\\
     &
     -\frac12 \partial_x^{-2}
           \Big[\big(\fe^{-i\tau\partial_x^2}\bar u^n\big)
              \cdot \fe^{i\tau\partial_x^2} \big( u^n\big)^2\Big]
                +\frac12 \fe^{i\tau\partial_x^2} \partial_x^{-2}
           \big[|u^n|^2 u^n\big]\notag\\
     &+ \partial_x^{-1}
          \Big[ \big(\fe^{i\tau\partial_x^2} u^n\big)
              \cdot \partial_x^{-1}
                \Big(\big|\fe^{i\tau\partial_x^2} u^n\big|^2 \Big)   \Big]
                -\fe^{i\tau\partial_x^2} \partial_x^{-1}
          \Big[ u^n
              \cdot \partial_x^{-1}
                \big(|u^n|^2 \big)   \Big].
\end{align}
Here we denote $\Pi_0(f)$ to be the zero mode of the function $f$, that is,
\begin{align}\label{def:ii}
\Pi_0(f)=\frac1{2\pi} \int_\T f(x)\,dx.
\end{align}
Moreover, $M_0, P_0$ are defined by 
\begin{align*}
M_0=M(u_0)= \Pi_0\left(|u_0|^2\right); \quad
P_0=\frac1{2\pi} \int_{\T} u_0 \partial_x\overline{u_0}\,d x=\Pi_0\left(u_0\partial_x \overline{u_0}\right).
\end{align*}


Now, we state the convergence theorem of the presented (semi-discretized) LRI method
given in \eqref{numerical} in one space dimension.
\begin{theorem}\label{result1}
Let $u^n$ be the numerical solution \eqref{numerical} of the  equation \eqref{model}
up to some fixed time $T>0$. Under assumption that $u_0\in H^{\gamma}(\bT)$ for some $\gamma>\frac{3}{2}$,  there exist constants $\tau_0,C>0$ such that for  any $0<\tau\leq\tau_0$,
\begin{equation}
  \|u(t_n,\cdot)-u^n\|_{H^\gamma}\leq C\tau,\quad n=0,1\ldots,\frac{T}{\tau},
\end{equation}
 where the constants $\tau_0$ and $C$ depend only on $T$ and $\|u\|_{L^\infty((0,T);H^{\gamma})}$.
\end{theorem}
Our theorem above improves the result in \cite{lownls} in one space dimension. In particular, our scheme does not loss any regularity in this case, which is the best one can expect in this sense. We believe that the idea used in this paper could be applied to the other models which will be studied in the forthcoming papers.

Based on the above theorem, we  also obtain $\frac12-$ order convergence
in $H^1(\T)$ with initial data in $ H^1(\bT)$, where $\frac12-$ denotes $\frac12-\epsilon$ for any arbitary small  $\epsilon>0$.
   It is practically reasonable to obtain lower convergence under lower regularity assumptions, because the accuracy order of the scheme in time and in space need to be rather equal.
\begin{corollary}\label{corollary}
Let $u^n$ be the numerical solution \eqref{numerical} of the  equation \eqref{model}
up to some fixed time $T>0$. Under assumption that $u_0\in H^1(\bT)$,  there exist constants $\tau_0,C>0$ such that for  any $0<\tau\leq\tau_0$,
\begin{equation}
  \|u(t_n,\cdot)-u^n\|_{H^1}\leq C\tau^{\frac12-},\quad n=0,1\ldots,\frac{T}{\tau},
\end{equation}
 where the constants $\tau_0$ and $C$ depend only on $T$ and $\|u\|_{L^\infty((0,T);H^1)}$.
\end{corollary}

Furthermore, we continue to pursue a  scheme such that it could be almost conserved in mass which  meanwhile requires as less regularity  as possible. To this purpose, we define a modified numerical scheme of \eqref{numerical} as follows. Let $\Psi$ be defined in \eqref{psi} and 
$$
F(U^n)=\Psi(U^n)-\fe^{i\tau\partial_x^2} U^n.
$$
Then we denote the functionals $G_1$, $G_2$  to be 
\begin{align}
 G_1(U)=&H(U)\fe^{i\tau\partial_x^2}U;\label{G1}\\
G_2(U)=&-\frac12 \big(H(U)\big)^2\fe^{i\tau\partial_x^2}U
      -M_0^{-1}\>H(U)\>  \mbox{Re}\> \Pi_0\Big(F(U)\>\fe^{-i\tau\partial_x^2}\bar U\Big) \>\fe^{i\tau\partial_x^2}U;
\end{align}
and
\begin{align}
H(U)=-M_0^{-1}\left[ \mbox{Re}\> \Pi_0\left(F(U)\>\fe^{-i\tau\partial_x^2}\bar U\right)
  +\frac12 \Pi_0\left(\big|F(U)\big|^2\right)\right].
\end{align}
Now the modified numerical scheme (NLRI) of \eqref{numerical} is defined by
 \begin{align}\label{scheme}
 U^{n+1}=\Psi(U^n)+G_1(U^n)+G_2(U^n),
 \end{align}
 for $ n=0,1,\cdots,\frac{T}{\tau}-1$, and $U^0=u_0$,

 Then we obtain that
\begin{theorem}\label{result}
Let $U^n$ be the numerical solution \eqref{scheme} of the  equation \eqref{model}
up to some fixed time $T>0$. Under assumption that $u_0\in H^{\gamma}(\bT)$ for some $\gamma>\frac{3}{2}$,  there exist constants $\tau_0,C>0$ such that for  any $0<\tau\leq\tau_0$ we have
\begin{equation}
  \|u(t_n,\cdot)-U^n\|_{H^\gamma}\leq C\tau,\quad n=0,1\ldots,\frac{T}{\tau}.
\end{equation}
 Moreover,
 \begin{align}\label{est:mass}
\left|M(U^n)-M(u_0)\right|\le  C\tau^5,
\end{align}
 where the constants $\tau_0$ and $C$ depend only on $T$ and $\|u\|_{L^\infty((0,T);H^{\gamma})}$.

 Furthermore, if $u_0\in H^1(\bT)$,  there exist constants $\tau_0,C>0$ such that for  any $0<\tau\leq\tau_0$,
\begin{equation}
  \|u(t_n,\cdot)-U^n\|_{H^1}\leq C\tau^{\frac12-},\quad n=0,1\ldots,\frac{T}{\tau},
\end{equation}
 where the constants $\tau_0$ and $C$ depend only on $T$ and $\|u\|_{L^\infty((0,T);H^1)}$.
\end{theorem}

To the best of our knowledge, this is the first attempt to consider the conservation laws of the numerical solution for the exponential-type integrators. 

\begin{remark}
In this paper, we present fifth-order mass convergence. However,
our method is also applicable to solve the equation \eqref{model} with arbitrary order  mass convergence by suitably adding  correction terms. 
\end{remark}

Now we slightly explain the key ingredient of the construction. Denote  $F$ to be
$$
u^{n+1}=\fe^{i\tau\partial_x^2}u^n+F(u^n),
$$
where $u^n$ is the numerical solution \eqref{numerical}.
Then we can find a functional $G$ such that
$$
\left\langle G(u^n),\fe^{i\tau\partial_x^2}u^n\right\rangle=-\left\langle F(u^n),\fe^{i\tau\partial_x^2}u^n \right\rangle,
$$
and
$$
\left\|G(u^n)\right\|_{H^\gamma}\le C\tau^2.
$$
The key point is that we only have the first-order estimate of $F^n(u^n)$ which reads
$$
\|F(u^n)\|_{H^\gamma}\le C\tau.
$$
Hence we can not choose $G(u^n)=-F(u^n)$ directly, and the cancellation in the $L^2$-inner product plays a great role in ensuring the second-order estimate of $G(u^n)$.
Based on the nice feature of $G$, we can modify the numerical solution $u^n$ and define the new scheme  by
$$
\tilde u^{n+1}=\fe^{i\tau\partial_x^2} \tilde u^n+F(\tilde u^n)+G(\tilde u^n).
$$
Then we can prove that
\begin{align*}
  \|u(t_n,\cdot)-\tilde u^n\|_{H^\gamma}\leq C\tau, \quad
\left|M(\tilde u^n)-M(u_0)\right|\le  C\tau^3.
\end{align*}
Repeating the same process, we can design a new scheme $U^n$ verifying the required accuracy  as in Theorem \ref{result}. More details will be given in Section \ref{sec:5}.


The paper is organized as follows.
In Section 2, we give some  notations and some useful lemmas.
In Section 3,
we give  the main process of the construction of the first-order scheme.
In Section 4, we devote to prove Theorem \ref{result1}.
Further discussion on the almost mass-conserved scheme is presented in Section 5.
Numerical confirmations are reported in Section \ref{sec:numerical} and conclusions are drawn in Section \ref{sec:conclusion}.

\section{Preliminary}
\subsection{Some notations}\label{subsec1}
We use $A\lesssim B$ or $B\gtrsim A$ to denote the statement that $A\leq CB$ for some absolute constant $C>0$ which may
vary from line to line but is independent of $\tau$ or $n$, and we denote $A\sim B$ for
$A\lesssim B\lesssim A$. We use $O(Y)$ to denote any quantity $X$ such that $X\lesssim Y$. Moreover, we denote $\langle \cdot,\cdot\rangle$ to be the $L^2$-inner product, that is
$$
\langle f,g \rangle=\mbox{Re} \int_\T f(x) \overline{g(x)}\,dx.
$$

The Fourier transform of a function $f$ on $\T$ is defined by
$$
\hat{f}_k=\frac{1}{2\pi}\int_{\T} e^{- i   kx }f( x )\,d x,
$$
and thus the Fourier inversion formula
$$
f( x )=\sum_{k\in \Z} e^{ i  kx } \hat{f}_k.
$$
Then the following usual properties of the Fourier transform hold:
\begin{eqnarray*}
 &\|f\|_{L^2(\T)}^2
 =2\pi \sum\limits_{k\in \Z}\big|\hat{f_k}\big|^2\quad \mbox{(Plancherel)}; \\
 & \widehat{(fg)}( k )=\sum\limits_{k_1\in\Z}
  \hat{f}_{k - k _1}\hat{g}_{k _1}  \quad \mbox{(Convolution)}.
\end{eqnarray*}
The Sobolev space $H^s(\T)$ for $s\geq0$ has the equivalent norm,
$$
\big\|f\big\|_{H^s(\T)}^2=
\big\|J^sf\big\|_{L^2(\T)}^2
=2\pi\sum_{k\in \Z}(1+ k ^2)^s |\hat{f}_k|^2,
$$
where we denote the operator
$$J^s=(1-\partial_{xx})^\frac s2.$$
Moreover, we denote $\partial_x^{-1}$ to be the operator defined by
\begin{equation}\label{def:px-1}
\widehat{\partial_x^{-1}f}( k )
=\Bigg\{ \aligned
    &(i k )^{-1}\hat{f}_k,\quad &\mbox{when }  k \ne 0,\\
    &0,\quad &\mbox{when }  k = 0.
   \endaligned
\end{equation}

We denote $T_m(M; v)$ to be a class of  qualities which is defined in the Fourier space by
\begin{align}\label{def:TmM}
\mathscr{F}T_m(M; v)(k)=O\bigg(\sup\limits_{t\in [0,T]}
\sum\limits_{ k = k _1+\cdots+ k _m}
| M( k, k _1,\cdots, k _m)|\>|\hat{v}_{ k _1}(t)| \cdots |\hat v_{ k _m}(t)|\bigg).
\end{align}
Here we regards $\bar v$ and $v$ as the same since there is no influence in the whole of analysis.
\subsection{Some preliminary estimates}\label{subsec3}
First, we will frequently apply the following Kato-Ponce inequality (simple version), which was originally proved in \cite{Kato-Ponce} and an important progress in the endpoint case was made in \cite{BoLi-KatoPonce, Li-KatoPonce} very recently.
\begin{lemma}\label{lem:kato-Ponce}(Kato-Ponce inequality) The following inequalities hold:
\begin{itemize}
  \item[(i)]
  For any $ \gamma>\frac 12$, $f,g\in H^{\gamma}$, then
\begin{align*}
\|J^\gamma (fg)\|_{L^2}\lesssim \|f\|_{H^\gamma}\|g\|_{H^{\gamma}}.
\end{align*}
  \item[(ii)]
For any  $\gamma\ge 0, \gamma_1>\frac 12$, $f\in H^{\gamma+\gamma_1},g\in H^{\gamma}$, then
\begin{align*}
\|J^\gamma (fg)\|_{L^2}\lesssim \|f\|_{H^{\gamma+\gamma_1}}\|g\|_{H^{\gamma}}.
\end{align*}
\end{itemize}
\end{lemma}

To prove our main result below, we need  the following two specific estimates.
\begin{lemma}\label{lem:An}
The following inequalities hold:
\begin{itemize}
\item[(i)]
Let $\gamma>\frac{3}{2}$, and $v\in {L^\infty((0,T);H^{\gamma})}$, then
$$
\big\|T_3(k^{-1}k_1k_2k_3;v)\big\|_{H^{\gamma}}
\lesssim \|v\|_{L^\infty((0,T);H^{\gamma})}^3.
$$
  \item[(ii)]
  Let $\gamma\geq1$, and $v\in {L^\infty((0,T);H^{\gamma})}$, then
  $$
 \big\|  T_3(k^{-1}k_1k_2^{\frac12-}k_3^{\frac12-};v)\big\|_{H^{\gamma}}
\lesssim \|v\|_{L^\infty((0,T);H^{\gamma})}^3.
  $$
 \item[(iii)]
  Let $\gamma\geq1$, and $v\in {L^\infty((0,T);H^{\gamma})}$, then
  $$
 \big|\mathscr{F}T_3(k_2k_3;v)(0)\big|
\lesssim \|v\|_{L^\infty((0,T);H^{\gamma})}^3.
  $$
 \item[(iv)]
Let $\gamma>\frac{1}{2}$, $m\geq1$ and $v\in {L^\infty((0,T);H^{\gamma})}$, then
$$
\big\|T_m(1;v)\big\|_{H^{\gamma}}
\lesssim \|v\|_{L^\infty((0,T);H^{\gamma})}^m.
$$
\end{itemize}
\end{lemma}
\begin{proof}
We assume that $\hat{v}_{k_j}(t), j=1,\cdots,m$ are positive for any $t\in [0,T]$, otherwise one may replace them by $|\hat{v}_{k_j}(t)|$.

(i)
Using the definition of $T_m(M)$ in \eqref{def:TmM},
 we have
\begin{align*}
|\mathscr{F}T_3(k^{-1}k_1k_2k_3;v)(k)|
\lesssim\sup\limits_{t\in [0,T]}
\sum\limits_{ k = k _1+ k _2+ k _3\ne 0} |k|^{-1}|k_1||k_2||k_3|
\>\hat{v}_{ k _1}(t) \hat{v}_{ k _2}(t)  \hat v_{ k _3}(t).
\end{align*}
By Plancherel's identity,  we get
\begin{align*}
\big\|T_3(k^{-1}k_1k_2k_3;v)\big\|_{H^{\gamma}}
&\lesssim
\Big\|\sum\limits_{ k = k _1+ k _2+ k _3\neq0}|k|^{\gamma-1}|k_1||k_2||k_3|
\>\hat{v}_{ k _1}(t) \hat{v}_{ k _2}(t)  \hat v_{ k _3}(t)\Big\|_{L^\infty((0,T);L^2)}\\
 & \lesssim \big\|(|\nabla| v)^3\big\|_{L^\infty((0,T);H^{\gamma-1})}.
\end{align*}
Therefore, by Lemma \ref{lem:kato-Ponce} (i), we obtain that for any $\gamma>\frac32$,
\begin{align*}
\big\|T_3(k^{-1}k_1k_2k_3;v)\big\|_{H^{\gamma}}
\lesssim \big\||\nabla| v\big\|_{L^\infty((0,T);H^{\gamma-1})}^3
\lesssim \|v\|_{L^\infty((0,T);H^{\gamma})}^3.
\end{align*}
(ii)
By the same argument to the proof of (i), we have
\begin{align*}
\big\|T_3(k^{-1}k_1k_2^{\frac12-}k_3^{\frac12-};v)\big\|_{H^{\gamma}}
&\lesssim
\Big\|\sum\limits_{ k = k _1+ k _2+ k _3\neq0}|k|^{\gamma-1}|k_1||k_2|^{\frac12-}|k_3|^{\frac12-}
\>\hat{v}_{ k _1}(t) \hat{v}_{ k _2}(t)  \hat v_{ k _3}(t)\Big\|_{L^\infty((0,T);L^2)}\\
 & \lesssim \big\|(|\nabla| v)(|\nabla|^{\frac12-}v)(|\nabla|^{\frac12-}v)
    \big\|_{L^\infty((0,T);H^{\gamma-1})}.
\end{align*}
Therefore, by Lemma \ref{lem:kato-Ponce} (ii), we obtain that for any $\gamma\geq1$,
\begin{align*}
\big\|T_3(k^{-1}k_1k_2^{\frac12-}k_3^{\frac12-};v)\big\|_{H^{\gamma}}
\lesssim  \big\|(|\nabla| v)(|\nabla|^{\frac12-}v)(|\nabla|^{\frac12-}v)
    \big\|_{L^\infty((0,T);H^{\gamma-1})}
\lesssim \|v\|_{L^\infty((0,T);H^{\gamma})}^3.
\end{align*}
(iii) From the definition of $T_m(M)$, we have  that for any $\gamma\ge 1$, 
\begin{align*}
|\mathscr{F}T_3(k_2k_3;v)(0)|
\lesssim &\sup\limits_{t\in [0,T]}
\sum\limits_{ 0 = k _1+\cdots+ k _m}
\>\hat{v}_{ k _1}(t) \>|k_2|\hat{v}_{ k _2}(t) \>|k_3|\hat{v}_{ k _3}(t) \\
\lesssim &\sup\limits_{t\in [0,T]}
\int_\T v \left(|\nabla |v\right)^2\,dx 
\lesssim \|v\|_{L^\infty((0,T);H^{\gamma})}^3.
\end{align*}
(iv)
Similarly, we have
\begin{align*}
|\mathscr{F}T_m(1;v)|\lesssim\sup\limits_{t\in [0,T]}
\sum\limits_{ k = k _1+\cdots+ k _m}
\>\hat{v}_{ k _1}(t) \cdots \hat v_{ k _m}(t).
\end{align*}
By Plancherel's identity and Lemma \ref{lem:kato-Ponce} (i), we obtain that for any $\gamma>\frac12$,
\begin{align*}
\big\|T_m(1;v)\big\|_{H^{\gamma}}
  \lesssim \|v^m\|_{L^\infty((0,T);H^{\gamma})}
\lesssim \|v\|_{L^\infty((0,T);H^{\gamma})}^m.
\end{align*}
Hence we get the desired result.
\end{proof}


\section{The first order scheme}

By Duhamel formula, we write
\begin{align*}
u(t_{n+1})=\fe^{i\tau\partial_x^2}u(t_n)
 -i\int_0^\tau \fe^{i\left(t_{n+1}-(t_n+s)\right)\partial_x^2}
   \big[|u(t_n+s)|^2u(t_n+s)\big]\,ds.
\end{align*}
Let $v(t):=\fe^{-it\partial_x^2}u(t)$, then
\begin{align}\label{solution}
v(t_{n+1})=v(t_n)-i\int_0^\tau \fe^{-i(t_n+s)\partial_x^2}
    \big[|\fe^{i(t_n+s)\partial_x^2}v(t_n+s)|^2\,\fe^{i(t_n+s)\partial_x^2}v(t_n+s)\big]\,ds.
\end{align}
Taking Fourier transform, we get
\begin{align*}
\hat v_{ k }(t_{n+1})
  =\hat v_{ k }(t_n)
     -i\int_0^\tau\sum\limits_{ k = k _1+ k _2+ k _3}\fe^{i(t_n+s)\phi}
        \>\widehat{ \bar v}_{ k _1}(t_n+s)\hat v_{ k _2}(t_n+s)\hat v_{ k _3}(t_n+s)\,ds.
\end{align*}
Here we denote $\hat v_{ k }(t)$ to be the $k$-th Fourier coefficients of $v(t)$, and the phase function
\begin{align*}
\phi=\phi( k, k _1, k _2, k _3)= k ^2+ k _1^2- k _2^2- k _3^2.
\end{align*}
By \eqref{solution}, we find that for any $s\in [0,\tau]$,
\begin{align}
v(t_n+s)=v(t_n)+\tau T_3(1;v).
\end{align}
Hence,  we have that
\begin{align}\label{v-1st-phi}
\hat v_{ k }(t_{n+1})=\hat v_{ k }(t_n)
    -i\sum\limits_{ k = k _1+ k _2+ k _3}\int_0^\tau \fe^{i(t_n+s)\phi}\,ds
       \>\widehat{ \bar v}_{ k _1}\hat v_{ k _2}\hat v_{ k _3}+\tau^2 \>\mathscr{F}{T_5(1;v)}(k).
\end{align}
Here and below, we denote $\hat v_{ k }$ to be $\hat v_{ k }(t_n)$ for short.

Now we split into the following two cases.

Case 1, $k=0$. Then by \eqref{v-1st-phi}, we get
\begin{align}\label{zeromode-1}
\hat v_0(t_{n+1})=&\hat v_0(t_n)-i\sum\limits_{k_1+k_2+k_3=0}\int_0^\tau \fe^{i(t_n+s)(k_1^2-k_2^2-k_3^2)}\,ds\>\widehat{ \bar v}_{k_1}\hat v_{k_2}\hat v_{k_3}
+\tau^2 \>\mathscr{F}T_5(1;v)(0).
\end{align}
Note that under the condition of $k_1+k_2+k_3=0$, we can transform the phase function
\begin{align*}
  k_1^2-k_2^2-k_3^2=2k_2 k_3.
\end{align*}
Therefore, we have
\begin{align}\label{v-1-2}
\int_0^\tau \left(\fe^{is(k_1^2-k_2^2-k_3^2)}-1\right)\,ds=\tau^2 O(k_2k_3).
\end{align}
According to \eqref{v-1-2}, we can freeze the phase function in the integrand in \eqref{zeromode-1} and obtain that 
\begin{align}\label{est:zeromode-1}
\hat v_0(t_{n+1})=&
       \hat v_0(t_n)-i\tau\sum\limits_{k_1+k_2+k_3=0} \fe^{it_n(k_1^2-k_2^2-k_3^2)}\>\widehat{ \bar v}_{k_1}\hat v_{k_2}\hat v_{k_3}
+\tau^2 \big(\mathscr{F}T_3(k_2k_3;v)(0)+\mathscr{F}T_5(1;v)(0)\big)\notag\\
=&
       \hat v_0(t_n)-i\tau \Pi_0\left(\left|\fe^{it_n \partial_{x}^2}v(t_n)\right|^2\fe^{it_n \partial_{x}^2}v(t_n)\right)
+\tau^2 \big(\mathscr{F}T_3(k_2k_3;v)(0)+\mathscr{F}T_5(1;v)(0)\big).
\end{align}


Case 2, $k\neq0$. For \eqref{v-1st-phi}, we only consider the term
\begin{equation}\label{v-2-1}
    -i\sum\limits_{k=k_1+k_2+k_3}\int_0^\tau \fe^{i(t_n+s)\phi}\,ds
       \>\widehat{ \bar v}_{k_1}\hat v_{k_2}\hat v_{k_3}.
\end{equation}
Note that
\begin{align*}
   1=\frac{(k_1+k_2)+(k_1+k_3)-k_1}{k}.
\end{align*}
Then by symmetry, it allows us to split \eqref{v-2-1} into two parts:
\begin{subequations}
\begin{align}
   &-2i\sum\limits_{k=k_1+k_2+k_3} \int_0^\tau\frac{k_1+k_2}{k} \fe^{i(t_n+s)\phi}\,ds
     \>\widehat{ \bar v}_{k_1}\hat v_{k_2}\hat v_{k_3}\label{v-2-2}\\
    &  +i\sum\limits_{k=k_1+k_2+k_3} \int_0^\tau\frac{k_1}{k} \fe^{i(t_n+s)\phi}\,ds
        \>\widehat{ \bar v}_{k_1}\hat v_{k_2}\hat v_{k_3}\label{v-2-3}.
\end{align}
\end{subequations}

For  \eqref{v-2-2}, we need the following equality:
If $k=k_1+k_2+k_3$, then
$$\phi=2(k_1+k_2)(k_1+k_3).$$
We note that if $k_1+k_3\neq0$, then
\begin{align}
\int_0^\tau\frac{k_1+k_2}{k} \fe^{i(t_n+s)\phi}\,ds
  =\frac{1}{2ik(k_1+k_3)}
     \big(\fe^{it_{n+1}\phi}-\fe^{it_n\phi}\big);
\end{align}
if $k_1+k_3=0$, then $\phi=0, k=k_2$ and thus
\begin{align}
\int_0^\tau\frac{k_1+k_2}{k} \fe^{i(t_n+s)\phi}\,ds
  =\tau\Big(\frac{k_1}{k}+1\Big).
\end{align}
Therefore,  we get
\begin{align*}
\eqref{v-2-2}
   =&-\sum\limits_{{\substack{k=k_1+k_2+k_3\\   k_1+k_3\ne 0}}} \frac{1}{k(k_1+k_3)}
     \big(\fe^{it_{n+1}\phi}-\fe^{it_n\phi}\big)
         \>\widehat{ \bar v}_{k_1}\hat v_{k_2}\hat v_{k_3}\\
   &\quad -2i\tau\sum\limits_{k_1+k_3= 0} \Big(\frac{k_1}{k}+1\Big)
         \>\widehat{ \bar v}_{k_1}\hat v_{k}\hat v_{k_3}.
\end{align*}
Now we need the following momentum conservation law:
\begin{align}
P\big(u(t)\big)&=\frac1{2\pi}\int_{\T} u(t)\bar u_ x (t)\,d x =P_0.\label{P-cons}
\end{align}
Note that by \eqref{mass} and \eqref{P-cons}, we have that
\begin{align*}
 -2i\tau\sum\limits_{k_1+k_3= 0} \Big(\frac{k_1}{k}+1\Big)
         \>\widehat{ \bar v}_{k_1}\hat v_{k}\hat v_{k_3}
      =&-\frac{i\tau}{\pi}  \int_\T \partial_x \bar u(t_n)\> u(t_n)\,dx\> (ik)^{-1}\hat v_{k}
         -\frac{i\tau}{\pi} \int_\T |u(t_n)|^2\,dx\> \hat v_{k} \\
      =&-2i \tau P_0 \>  (ik)^{-1}\hat v_{k}  -2i\tau M_0  \>\hat v_{k}.
\end{align*}
Therefore, we further obtain
\begin{align}\label{est:v-2-2}
\eqref{v-2-2}
   =&-\sum\limits_{{\substack{k=k_1+k_2+k_3\\   k_1+k_3\ne 0}}} \frac{1}{k(k_1+k_3)}
     \big(\fe^{it_{n+1}\phi}-\fe^{it_n\phi}\big)
         \>\widehat{ \bar v}_{k_1}\hat v_{k_2}\hat v_{k_3}
    -2i\tau P_0   \> (ik)^{-1}\hat v_{k} -2i\tau M_0  \>\hat v_{k}.
\end{align}

For \eqref{v-2-3},
we note that it can not be integrated in the physical space exactly.
Now we need the following two equalities:
If $k=k_1+k_2+k_3$, then
\begin{align}
&\phi(k,k_1,k_2,k_3)=2kk_1+2k_2 k_3;\label{phase-1}\\
&2kk_1=k^2+k_1^2-(k_2+k_3)^2.\label{phase-2}
\end{align}
Putting \eqref{phase-1} into \eqref{v-2-3}, we decompose \eqref{v-2-3} into two subparts again:
\begin{subequations}\label{v-2-45}
\renewcommand{\theequation}
{\theparentequation-\arabic{equation}}
\begin{align}
    &i\sum\limits_{k=k_1+k_2+k_3} \int_0^\tau \frac{k_1}{k}
      \fe^{it_n\phi} \big(\fe^{2isk_2k_3}-1\big)\fe^{2iskk_1}\,ds
           \>\widehat{ \bar v}_{k_1}\hat v_{k_2}\hat v_{k_3}\label{v-2-4}\\
    &+ i\sum\limits_{k=k_1+k_2+k_3} \int_0^\tau \frac{k_1}{k}
         \fe^{it_n\phi} \fe^{2iskk_1}\,ds
           \>\widehat{ \bar v}_{k_1}\hat v_{k_2}\hat v_{k_3}\label{v-2-5}.
\end{align}
\end{subequations}
For \eqref{v-2-4}, applying the inequality
$$
\big|\fe^{2isk_2k_3}-1\big|\le 2\tau |k_2||k_3|,\quad \mbox{ for any } s\in [0,\tau],
$$
we have that
 \begin{equation}\label{equal}
 \eqref{v-2-4}= \tau^2 \mathscr{F}T_3(k^{-1}k_1k_2k_3;v)(k).
 \end{equation}

For \eqref{v-2-5}, from \eqref{phase-2}
we have
$$
 \frac{k_1}{k}\int_0^\tau\fe^{2iskk_1}\,ds= \frac{1}{2ik^2}\left(\fe^{i\tau \big(k^2+k_1^2-(k_2+k_3)^2\big)}-1\right),
$$
we get
\begin{align*}
   \eqref{v-2-5}
       =\sum\limits_{k=k_1+k_2+k_3}  \frac{1}{2k^2}
              \fe^{it_n\phi}\left(\fe^{i\tau \big(k^2+k_1^2-(k_2+k_3)^2\big)}-1\right)
               \>\widehat{ \bar v}_{k_1}\hat v_{k_2}\hat v_{k_3}.
\end{align*}
Therefore, collecting the two finding in \eqref{v-2-45}, we obtain
\begin{align}\label{est:v-2-3}
\eqref{v-2-3}=\sum\limits_{k=k_1+k_2+k_3}  \frac{1}{2k^2}
              \fe^{it_n\phi}\left(\fe^{i\tau \big(k^2+k_1^2-(k_2+k_3)^2\big)}-1\right)
               \>\widehat{ \bar v}_{k_1}\hat v_{k_2}\hat v_{k_3}+ \tau^2 \mathscr{F}T_3(k^{-1}k_1k_2k_3;v)(k).
\end{align}

Combining with \eqref{est:v-2-2} and \eqref{est:v-2-3}, we derive that when $k\ne 0$,
\begin{align}\label{est:v-n0}
\hat v_k(t_{n+1})=&\hat v_k(t_n)
    -\sum\limits_{{\substack{k=k_1+k_2+k_3\\   k_1+k_3\ne 0}}} \frac{1}{k(k_1+k_3)}
     \big(\fe^{it_{n+1}\phi}-\fe^{it_n\phi}\big)
         \>\widehat{ \bar v}_{k_1}\hat v_{k_2}\hat v_{k_3}\notag\\
        &
           -2i\tau P_0\>  (ik)^{-1}\hat v_{k}  -2i\tau M_0\>\hat v_{k}
        \notag\\
   &
         +  \sum\limits_{k=k_1+k_2+k_3}  \frac{1}{2k^2}
              \fe^{it_n\phi} \left(\fe^{i\tau \big(k^2+k_1^2-(k_2+k_3)^2\big)}-1\right)
               \>\widehat{ \bar v}_{k_1}\hat v_{k_2}\hat v_{k_3}\notag\\
        & +\tau^2 \Big(\mathscr{F}T_3(k^{-1}k_1k_2k_3;v)+\mathscr{F}T_5(1;v)\Big)(k).
\end{align}

Now together with \eqref{est:zeromode-1} and \eqref{est:v-n0}, and using the formula 
$$
f-2i\tau M_0  f-2i\tau P_0\>\partial_x^{-1} f=\fe^{-2i\tau M_0-2i\tau P_0\partial_x^{-1}}f+\tau^2T_1(1;f),
$$
and  the inverse Fourier transform, we get
\begin{align}\label{vtn+1-app}
v(t_{n+1})=
&\Phi^n\big(v(t_n)\big)+\tau^2 \Big(T_1(1;v)+\mathscr{F}T_3(k_2k_3;v)(0)+T_3(k^{-1}k_1k_2k_3;v)+T_5(1;v)\Big),
\end{align}
where $\Phi^n$ is defined by
\begin{align}\label{Phi-def}
\Phi^n(f)
=&\fe^{-2i\tau M_0-2i\tau P_0\partial_x^{-1}}f+2i\tau M_0  \Pi_0(f)-i\tau \Pi_0\left(\left|\fe^{it_n \partial_{x}^2}f\right|^2\fe^{it_n \partial_{x}^2}f\right)\notag\\
     &+\fe^{-it_{n+1}\partial_x^2} \partial_x^{-1}
          \Big[ \big(\fe^{it_{n+1}\partial_x^2} f\big)
              \cdot \partial_x^{-1}
                \Big(\big|\fe^{it_{n+1}\partial_x^2} f\big|^2 \Big)   \Big]
                -\fe^{-it_n\partial_x^2} \partial_x^{-1}
          \Big[ \big(\fe^{it_n\partial_x^2} f\big)
              \cdot \partial_x^{-1}
                \Big(\big|\fe^{it_n\partial_x^2}f\big|^2 \Big)   \Big]\notag\\
     &-\frac12 \fe^{-it_{n+1}\partial_x^2} \partial_x^{-2}
           \Big[\big(\fe^{-it_{n+1}\partial_x^2}\bar f\big)
              \cdot \fe^{i\tau\partial_x^2}
                 \big(\fe^{it_n\partial_x^2} f\big)^2\Big]
                 +\frac12 \fe^{-it_n\partial_x^2} \partial_x^{-2}
           \Big[\big|\fe^{it_n\partial_x^2} f\big|^2\>\fe^{it_n\partial_x^2} f\Big].
\end{align}
Accordingly, we define the numerical solution of \eqref{model} by
\begin{align}\label{NuSo-NLS}
v^{n+1}=\Phi^n\big(v^n\big), \quad n=0,1\ldots,\frac{T}{\tau}-1;\quad v^0=u_0.
\end{align}
Let $u^n:=\fe^{it_n\partial_x^2}v^n$, this gives the scheme \eqref{numerical}--\eqref{psi} and thus finishes the construction of the numerical scheme.


\section{The proofs of Theorem \ref{result1} and Corollary  \ref{corollary}}\label{proof}
\subsection{The proof of the Theorem \ref{result1}}
From \eqref{NuSo-NLS}, we have
\begin{align*}
v(t_{n+1})-v^{n+1}=&v(t_{n+1})-\Phi^n\big(v(t_n)\big)+ \Phi^n\big(v(t_n)\big)-\Phi^n\big(v^n\big)\\
          \triangleq &\mathcal{L}^n+\Phi^n\big(v(t_n)\big)-\Phi^n\big(v^n\big),
\end{align*}
where $\mathcal{L}^n=v(t_{n+1})-\Phi^n\big(v(t_n)\big)$.

Furthermore, from \eqref{vtn+1-app} we get
$$
\mathcal{L}^n=\tau^2\Big(T_1(1;v)+\mathscr{F}T_3(k_2k_3;v)(0)+T_3(k^{-1}k_1k_2k_3;v)+T_5(1;v)\Big).
$$
Then from  Lemma \ref{lem:An}, we have
\begin{align}\label{est:ln}
\|\mathcal{L}^n\|_{H^\gamma}\leq C\tau^2,
\end{align}
where the constant  $C$ depends only  on $\|u\|_{L^\infty((0,T);H^{\gamma})}$.

Note that   $\Phi^n(f)$ defined in \eqref{Phi-def} can be read as the following integral form:
\begin{align*}
\Phi^n(f)=&f-i\tau \Pi_0\left(\left|\fe^{it_n \partial_{x}^2}f\right|^2\fe^{it_n \partial_{x}^2}f\right)\\
 & -2i\int_0^\tau\fe^{-i(t_n+s)\partial_x^2}\partial_x^{-1}
          \Big( \fe^{-i(t_n+s)\partial_x^2}\partial_x\bar f
    \cdot  \big(\fe^{i(t_n+s)\partial_x^2}f\big)^2\Big)ds\\
& -2i\int_0^\tau\fe^{-i(t_n+s)\partial_x^2}\partial_x^{-1}
            \Big(\big|\fe^{i(t_n+s)\partial_x^2}f\big|^2
                \fe^{i(t_n+s)\partial_x^2}\partial_xf\Big)ds\\
& +i\int_0^\tau\fe^{-i(t_n+s)\partial_x^2}\partial_x^{-1}
          \Big( \fe^{-i(t_n+s)\partial_x^2}\partial_x\bar f
    \cdot\fe^{is\partial_x^2}  \big(\fe^{it_n\partial_x^2}f\big)^2\Big)ds\\
  &+2i\tau\big( M(f)-M(v(t_n)\big) f+2i\tau\big( P(f)-P(v(t_n)\big)\partial_x^{-1} f\\
           &+\Big[\fe^{-2i\tau\left(M_0+P_0\partial_x^{-1}\right)}-1+2i\tau\left(M_0+P_0\partial_x^{-1}\right)\Big] f.
\end{align*}
Therefore, we obtain
\begin{align*}
\Phi^n\big(v(t_n)\big)-\Phi^n\big(v^n\big)
     =&v(t_n)-v^n+\Phi^n_1+\Phi^n_2+\Phi^n_3+\Phi^n_4+\Phi^n_5,
 \end{align*}
 where
 \begin{align*}
  \Phi^n_1=& -i\tau \Pi_0\left(\left|\fe^{it_n \partial_{x}^2}v(t_n)\right|^2\fe^{it_n \partial_{x}^2}v(t_n)\right)
      +i\tau \Pi_0\left(\left|\fe^{it_n \partial_{x}^2}v^n\right|^2\fe^{it_n \partial_{x}^2}v^n\right);\\
   \Phi^n_2=& -2i\int_0^\tau\fe^{-i(t_n+s)\partial_x^2}\partial_x^{-1}
          \Big( \fe^{-i(t_n+s)\partial_x^2}\partial_x\bar v(t_n)
         \cdot  \big(\fe^{i(t_n+s)\partial_x^2}v(t_n)\big)^2\Big)ds\\
    &+2i\int_0^\tau\fe^{-i(t_n+s)\partial_x^2}\partial_x^{-1}
          \Big( \fe^{-i(t_n+s)\partial_x^2}\partial_x\bar v^n
    \cdot  \big(\fe^{i(t_n+s)\partial_x^2}v^n\big)^2\Big)ds;\\
      \Phi^n_3=& -2i\int_0^\tau\fe^{-i(t_n+s)\partial_x^2}\partial_x^{-1}
            \Big(\big|\fe^{i(t_n+s)\partial_x^2}v(t_n)\big|^2
                \fe^{i(t_n+s)\partial_x^2}\partial_xv(t_n)\Big)ds\\
                  & +2i\int_0^\tau\fe^{-i(t_n+s)\partial_x^2}\partial_x^{-1}
            \Big(\big|\fe^{i(t_n+s)\partial_x^2}v^n\big|^2
                \fe^{i(t_n+s)\partial_x^2}\partial_xv^n\Big)ds;\\
    \Phi^n_4=&i\int_0^\tau\fe^{-i(t_n+s)\partial_x^2}\partial_x^{-1}
          \Big( \fe^{-i(t_n+s)\partial_x^2}\partial_x\bar v(t_n)
    \cdot\fe^{is\partial_x^2}  \big(\fe^{it_n\partial_x^2}v(t_n)\big)^2\Big)ds \\
    &-i\int_0^\tau\fe^{-i(t_n+s)\partial_x^2}\partial_x^{-1}
          \Big( \fe^{-i(t_n+s)\partial_x^2}\partial_x\bar v^n
    \cdot\fe^{is\partial_x^2}  \big(\fe^{it_n\partial_x^2}v^n\big)^2\Big)ds;\\
    \Phi^n_5=&-2i\tau\big( M(v^n)-M(v(t_n)\big) v^n-2i\tau\big( P(v^n)-P(v(t_n)\big)\partial_x^{-1}v^n \\
    &+\left[\fe^{-2i\tau\left(M_0+P_0\partial_x^{-1}\right)}-1+2i\tau\left(M_0+P_0\partial_x^{-1}\right)\right]\left(v(t_n)-v^n\right).
\end{align*}
Next we estimate the above terms.
$\Phi^n_1$ can be divided into three parts
\begin{align*}
\Phi^n_1=&-i\tau \Pi_0\left[
            \big|\fe^{it_n\partial_x^2}v(t_n)\big|^2
                \fe^{it_n\partial_x^2}\big(v(t_n)-v^n\big)\right]\\
&-i\tau \Pi_0\left[\Big(\big|\fe^{it_n\partial_x^2}v(t_n)\big|^2-\big|\fe^{it_n\partial_x^2}v^n\big|^2\Big)
                \fe^{it_n\partial_x^2}v(t_n)\right]\\
&-i\tau \Pi_0\left[\Big(\big|\fe^{it_n\partial_x^2}v(t_n)\big|^2-\big|\fe^{it_n\partial_x^2}v^n\big|^2\Big)
                \fe^{it_n\partial_x^2}\big(v^n-v(t_n)\big)\right].
\end{align*}
Then by the  H\"older and Sobolev inequalities, we obtain that
\begin{align}\label{stability1}
\|\Phi^n_1\|_{H^\gamma}\leq C|\Phi^n_1|\leq C\tau\big( \|v^n-v(t_n)\|_{H^\gamma}+\|v^n-v(t_n)\|_{H^\gamma}^3\big).
\end{align}
Similarly, by Lemma \ref{lem:kato-Ponce}, we have that for any $\gamma\geq1$,
\begin{align}\label{stability2}
\|\Phi^n_j\|_{H^\gamma}\leq C\tau\big( \|v^n-v(t_n)\|_{H^\gamma}+\|v^n-v(t_n)\|_{H^\gamma}^3\big),
 \quad j=2,\cdots, 5.
\end{align}
Therefore for any $\gamma\geq1$,
\begin{align}\label{stability3}
\|\Phi^n\big(v(t_n)\big)-\Phi^n\big(v^n\big)\|_{H^\gamma}
   \leq &(1+C\tau) \|v^n-v(t_n)\|_{H^\gamma}
        +C\tau\|v^n-v(t_n)\|_{H^\gamma}^3.
   \end{align}
Combining the above estimates, we conclude that
\begin{align*}
\|v(t_{n+1})-v^{n+1}\|_{H^\gamma}
   \leq &C\tau^2+(1+C\tau) \|v^n-v(t_n)\|_{H^\gamma}
        +C\tau\|v^n-v(t_n)\|_{H^\gamma}^3.
\end{align*}
By iteration and Gronwall's inequalities, we get
\begin{align*}
\big\|v(t_{n+1})-v^{n+1}\big\|_{H^\gamma}
\le C\tau^2\sum\limits_{j=0}^n(1+C\tau)^j\le C \tau,\quad n=0,1,\ldots,\frac{T}{\tau}-1.
\end{align*}
This finishes the proof of the convergence result.

\subsection{The proof of the Corollary \ref{corollary}}
For \eqref{v-2-4}, applying the inequality
$$
\big|\fe^{2isk_2k_3}-1\big|\le 2\tau^{\frac12-} |k_2|^{\frac12-}|k_3|^{\frac12-},
   \quad \mbox{ for any } s\in [0,\tau],
$$
then we can replace \eqref{equal} with
 \begin{equation}
 \eqref{v-2-4}= \tau^{\frac32-} \mathscr{F}T_3(k^{-1}k_1k_2^{\frac12-}k_3^{\frac12-};v)(k).
 \end{equation}
 Therefore, we have
\begin{align}\label{v-int2}
v(t_{n+1})=
&\Phi^n\big(v(t_n)\big)+\tau^{\frac32-} T_3(k^{-1}k_1k_2^{\frac12-}k_3^{\frac12-};v)+\tau^2\Big(T_1(1;v)+\mathscr{F}T_3(k_2k_3;v)(0)+T_5(1;v)\Big).
\end{align}
Similarly as before,
\begin{align*}
v(t_{n+1})-v^{n+1}=&v(t_{n+1})-\Phi^n\big(v(t_n)\big)+ \Phi^n\big(v(t_n)\big)-\Phi^n\big(v^n\big)\\
          \triangleq &\mathcal{\tilde L}^n+\Phi^n\big(v(t_n)\big)-\Phi^n\big(v^n\big),
\end{align*}
where
$$
\mathcal{\tilde L}^n=\tau^{\frac32-} T_3(k^{-1}k_1k_2^{\frac12-}k_3^{\frac12-};v)+\tau^2\Big(T_1(1;v)+\mathscr{F}T_3(k_2k_3;v)(0)+T_5(1;v)\Big).
$$
Then from  Lemma \ref{lem:An}, we obtain that 
\begin{align}\label{est:err2}
\|\mathcal{\tilde L}^n\|_{H^1}\leq C\tau^{\frac32-},
\end{align}
where the  constant $C>0$
depending only on  $\|u\|_{L^\infty((0,T);H^1)}$.
This together with \eqref{stability3} yields
\begin{align*}
\|v(t_{n+1})-v^{n+1}\|_{H^1}
   \leq &C\tau^{\frac32-}+(1+C\tau) \|v^n-v(t_n)\|_{H^1}
        +C\tau\|v^n-v(t_n)\|_{H^1}^3.
\end{align*}
By iteration and Gronwall's inequalities, we get
\begin{align*}
\big\|v(t_{n+1})-v^{n+1}\big\|_{H^1}
\le C\tau^{\frac32-}\sum\limits_{j=0}^n(1+C\tau)^j\le C \tau^{\frac12-},\quad n=0,1,\ldots,\frac{T}{\tau}-1.
\end{align*}
Hence, we get  the desired convergence result.

\vskip 25pt

\section{Further discussion on the almost mass-conserved scheme}\label{sec:5}

 Let $V^n:=\fe^{-it_n\partial_x^2}U^n$. Accordingly, from \eqref{G1}--\eqref{scheme}, we have that 
 \begin{align}\label{V}
 V^{n+1}=V^n+\tilde F^n(V^n)+\tilde G_1^n(V^n)+\tilde G_2^n(V^n),
 \end{align}
 where $\Phi^n$ is defined in \eqref{Phi-def},
 $$
 \tilde F^n(V^n)=\Phi^n(V^n)-V^n,
 $$
and the functionals $\tilde G_1^n$, $\tilde G_2^n$  are given by
\begin{align}
\tilde G^n_1(V)=\tilde H^n(V)V ;
\quad
\tilde G^n_2(V)=-\frac12 \big(\tilde H^n(V)\big)^2V
      -M(u_0)^{-1}\>\tilde H^n(V)\>  \big\langle \tilde F^n(V), V\big\rangle \>V\label{def:G1},
\end{align}
and
\begin{align}
\tilde H^n(V)=-M(u_0)^{-1}\left( \big\langle \tilde F^n(V), V\big\rangle
  +\frac12 \big\|\tilde F^n(V)\big\|_{L^2}^2\right).
\end{align}

The proof of  Theorem  \ref{result} depends on the following key lemmas.
\begin{lemma}\label{def:lemG1}
Let $\tilde G^n_1$ be defined in  \eqref{def:G1}, then the following inequalities holds:
\begin{itemize}
\item[(i)]
If  $\gamma>\frac32$,  then there exists some constant  $C=C(\|V\|_{H^\gamma}, \|u_0\|_{L^2})>0$  such that 
\begin{align*}
2 \big\langle \tilde G_1^n(V), V\big\rangle
         +2\big\langle &\tilde F^n(V), V\big\rangle
+\big\|\tilde F^n(V)\big\|_{L^2}^2
  \leq C\tau^2\big|M(V)-M(u_0)\big|,
\end{align*}
moreover,
$$
\big\|\tilde G_1^n(V)\big\|_{H^\gamma}\leq C\tau^2.
$$
\item[(ii)]
If  $\gamma\geq 1$, then  there exists some constant  $C=C(\|V\|_{H^\gamma}, \|u_0\|_{L^2})>0$  such that 
\begin{align*}
2 \big\langle \tilde G_1^n(V), V\big\rangle
         +2\big\langle &\tilde F^n(V), V\big\rangle
+\big\|\tilde F^n(V)\big\|_{L^2}^2
  \leq C\tau^{\frac32-}\big|M(V)-M(u_0)\big|,
\end{align*}
moreover,
$$
\big\|\tilde G_1^n(V)\big\|_{H^\gamma}\leq C\tau^{\frac32-}.
$$
\end{itemize}
\end{lemma}
\begin{proof}
(i)
According to \eqref{v-1st-phi} and \eqref{vtn+1-app}, we find that
\begin{align}\label{Fn}
\tilde F^n\big(V\big)
    =&-i\int_0^\tau \fe^{-i(t_n+s)\partial_x^2}
    \left[|\fe^{i(t_n+s)\partial_x^2}V|^2\,\fe^{i(t_n+s)\partial_x^2}V\right]\,ds\notag\\
    &\quad  +2i\tau\big( M(V)-M(v(t_n)\big) V+2i\tau\big( P(V)-P(v(t_n)\big)\partial_x^{-1}V\notag\\
     &\quad -\tau^2\Big(T_1(1;V)+\mathscr{F}T_3(k_2k_3;V)(0)+T_3(k^{-1}k_1k_2k_3;V)\Big).
\end{align}
Then we have
\begin{align*}
  \big\langle \tilde F^n(V), V\big\rangle
 =&  \Big\langle -i\int_0^\tau \fe^{-i(t_n+s)\partial_x^2}
    \big[|\fe^{i(t_n+s)\partial_x^2}V|^2\,\fe^{i(t_n+s)\partial_x^2}V\big]\,ds, \ V\Big\rangle\\
    &+2\tau \Big\langle i\big( M(V)-M(v(t_n)\big) V+i\big( P(V)-P(v(t_n)\big)\partial_x^{-1}V,V\Big\rangle\\
   & -\tau^2   \Big\langle T_1(1;V)+\mathscr{F}T_3(k_2k_3;V)(0)+T_3(k^{-1}k_1k_2k_3;V), \ V\Big\rangle.
\end{align*}
The first term is obviously equal to 0. We claim that the second term is also equal to 0. Indeed,  since 
$$
iP(f)=-\frac1{2\pi}\mbox{Im}  \int_{\T} f \partial_x\bar{f}\,d x\in \R,
$$
we find that 
\begin{align*}
\Big\langle i\big( M(V)-&M(v(t_n)\big) V+i\big( P(V)-P(v(t_n)\big)\partial_x^{-1}V,V\Big\rangle\\
=&\big( M(V)-M(v(t_n)\big)\langle i V,V\rangle+i\big( P(V)-P(v(t_n)\big)\langle\partial_x^{-1}V,V\rangle
=0.
\end{align*}
Then we get
\begin{align*}
  \big\langle \tilde F^n(V), V\big\rangle=&
       -\tau^2\big\langle T_1(1;V)+\mathscr{F}T_3(k_2k_3;V)(0)+T_3(k^{-1}k_1k_2k_3;V), \ V\big\rangle.
\end{align*}
From Lemma \ref{lem:An} (i) (iii), we obtain
\begin{align}\label{est:T3}
\big\| T_1(1;V)\|_{L^2}+\big|\mathscr{F}T_3(k_2k_3;V)(0)\big|+\big\|T_3(k^{-1}k_1k_2k_3;V)\big\|_{L^2} 
\lesssim\big\|V\big\|_{H^\gamma}+ \big\|V\big\|_{H^\gamma}^3.
\end{align}
This implies that
\begin{align}\label{est:H1}
\big|  \big\langle \tilde F^n(V), V\big\rangle\big|\lesssim \tau^2\big(\|V\|_{H^{\gamma}}^2+ \|V\|_{H^{\gamma}}^4\big).
\end{align}
By  \eqref{Fn} and  \eqref{est:T3}, we have
\begin{align}\label{est:fn}
\big\|\tilde F^n(V)\big\|_{L^2}^2\lesssim \tau^2\big(\|V\|_{H^{\gamma}}^2+ \|V\|_{H^{\gamma}}^6\big).
\end{align}
Hence, there exists $C>0$ depending only on $\|V\|_{H^\gamma}$ and $\|u_0\|_{L^2}$, such that
\begin{align}\label{est:H}
\big|\tilde H^n(V)\big|\leq C\tau^2.
\end{align}
This yields that
$$
\big\|\tilde G_1^n(V)\big\|_{H^\gamma}\leq C\tau^2.
$$

Furthermore, we have
\begin{align}\label{lemma1:formula}
2 \big\langle \tilde G_1^n(V), V\big\rangle
         +2\big\langle \tilde F^n(V), V\big\rangle
+\big\|\tilde F^n(V)\big\|_{L^2}^2
  =2\tilde H^n(V)\big(M(V)-M(u_0)\big).
\end{align}
From  \eqref{est:H}, the above equality is controlled by
$
C\tau^2\left|M(V)-M(u_0)\right|.
$

(ii)
According to \eqref{v-1st-phi} and \eqref{v-int2}, we can replace $\tau^2T_3(k^{-1}k_1k_2k_3;V)$ in \eqref{Fn}  by 
$$\tau^{\frac32-} T_3(k^{-1}k_1k_2^{\frac12-}k_3^{\frac12-};V).$$ 
Then arguing similarly as in the proof of (i) and applying Lemma \ref{lem:An} (ii) (iii) instead, it infers that 
\begin{align}\label{est:H2}
\big|  \big\langle \tilde F^n(V), V\big\rangle\big|\leq C \tau^{\frac32-};
\quad
\big|\tilde H^n(V)\big|\leq C\tau^{\frac32-}.
\end{align}
Hence we have for any $\gamma\geq1$
$$
\big\|\tilde G_1^n(V)\big\|_{H^\gamma}\leq C\tau^{\frac32-}.
$$
Moreover, by \eqref{lemma1:formula} and \eqref{est:H2}, we obtain
\begin{align*}
2 \big\langle \tilde G_1^n(V), V\big\rangle
         +2\big\langle &\tilde F^n(V), V\big\rangle
+\big\|\tilde F^n(V)\big\|_{L^2}^2
  \leq C\tau^{\frac32-}\big|M(V)-M(u_0)\big|.
\end{align*}
This finishes the proof of the lemma.
\end{proof}

\begin{lemma}\label{def:lemG2}
Let the functionals  $\tilde G_1^n$, $\tilde G^n_2$ be defined in \eqref{def:G1}.
\begin{itemize}
\item[(i)]
If $\gamma>\frac32$, then there exists some constant  $C=C(\|V\|_{H^\gamma}, \|u_0\|_{L^2})>0$  such that 
\begin{align*}
2 \big\langle \tilde G_2^n(V), V\big\rangle
       +2 \big\langle \tilde F^n(V), \tilde G_1^n(V)\big\rangle
         +\big\|\tilde G_1^n(V)\big\|_{L^2}^2
  \leq C\tau^4\big|M(V)-M(u_0)\big|,
\end{align*}
moreover,
$$
\big\|\tilde G_2^n(V)\big\|_{H^\gamma}\leq C\tau^4.
$$

\item[(ii)]
If $\gamma\geq1$, then there exists some constant  $C=C(\|V\|_{H^\gamma}, \|u_0\|_{L^2})>0$  such that 
\begin{align*}
2 \big\langle \tilde G_2^n(V), V\big\rangle
       +2 \big\langle \tilde F^n(V), \tilde G_1^n(V)\big\rangle
         +\big\|\tilde G_1^n(V)\big\|_{L^2}^2
  \leq C\tau^{3-}\big|M(V)-M(u_0)\big|,
\end{align*}
moreover,
$$
\big\|\tilde G_2^n(V)\big\|_{H^\gamma}\leq C\tau^{3-}.
$$
\end{itemize}

\end{lemma}

\begin{proof}
(i)
From the definition of $\tilde G_1^n(V)$  in \eqref{def:G1}, we find
\begin{align*}
2 \big\langle \tilde F^n(V), \tilde G_1^n(V)\big\rangle
  =2\tilde H^n(V)  \big\langle \tilde F^n(V), V\big\rangle;
  \quad
  \big\|\tilde G_1^n(V)\big\|_{L^2}^2
    =\big(\tilde H^n(V)\big)^2M(V).
\end{align*}
Note that
\begin{align}\label{formula-lemma2}
2 \big\langle \tilde G_2^n(V)&, V\big\rangle
       +2 \big\langle \tilde F^n(V), \tilde G_1^n(V)\big\rangle
         +\big\|\tilde G_1^n(V)\big\|_{L^2}^2\notag\\
  =&2\tilde H^n(V)  \big\langle \tilde F^n(V), V\big\rangle
  \frac{M(u_0)-M(V)}{M(u_0)}.
\end{align}
By \eqref{est:H1} and
\eqref{est:H}, the above equality
is controlled by
$$
 C\tau^4\big|M(V)-M(u_0)\big|.
$$
Moreover, by \eqref{est:H1} and
\eqref{est:H} again, we obtain
$$
\big\|\tilde G_2^n(V)\big\|_{H^\gamma}\leq C\tau^4.
$$

(ii)
From \eqref{est:H2} and \eqref{formula-lemma2},
we get directly
\begin{align*}
2 \big\langle \tilde G_2^n(V), V\big\rangle
       +2 \big\langle \tilde F^n(V), \tilde G_1^n(V)\big\rangle
         +\big\|\tilde G_1^n(V)\big\|_{L^2}^2
  \leq C\tau^{3-}\big|M(V)-M(u_0)\big|,
\end{align*}
and
$$
\big\|\tilde G_2^n(V)\big\|_{H^\gamma}\leq C\tau^{3-}.
$$
We obtain the conclusion of the lemma.
\end{proof}

\begin{proof}[\bf Proof of  Theorem \ref{result}]
Since $V^n=\fe^{-it_n\partial_x^2}U^n$, $v(t_n)=\fe^{-it_n\partial_x^2}u(t_n)$,
 we only need to prove the conclusion of Theorem \ref{result} holds for $V^n$ and $v(t_n)$.

From \eqref{V}, we have
\begin{align}
 V^{n+1}=\Phi^n(V^n)+\tilde G^n_1(V^n)+\tilde G^n_2(V^n),
 \end{align}
 where $\Phi^n(V^n)$ is defined in \eqref{Phi-def}. Then
 \begin{align*}
 V^{n+1}-v(t_{n+1})=&\Phi^n(V^n)-\Phi^n(v(t_n))+\tilde G^n_1(V^n)-\tilde G^n_1(v(t_n))
                    +\tilde G^n_2(V^n)-\tilde G^n_2(v(t_n))\\
                    &+\Phi^n(v(t_n))-v(t_{n+1})+\tilde G^n_1(v(t_n))+\tilde G^n_2(v(t_n)).
 \end{align*}
From the estimate on the functional $\Phi^n$ in \eqref{stability3}, we obtain for $\gamma\geq1$
\begin{align*}
\big\|\Phi^n\big(v(t_n)\big)-\Phi^n\big(V^n\big)\big\|_{H^\gamma}
   \leq &(1+C\tau) \big\|V^n-v(t_n)\big\|_{H^\gamma}
        +C\tau\big\|V^n-v(t_n)\big\|_{H^\gamma}^3.
   \end{align*}

 For the term $\tilde G^n_j(V^n)-\tilde G^n_j(v(t_n)), j=1,2$, the similar treatment as in Section \ref{proof},
  we get
%
\begin{align*}
 \big\|\tilde G^n_1(V^n)-\tilde G^n_1(v(t_n))\big\|_{H^\gamma}
     \leq C\tau\Big(\big\|V^n-v(t_n)\big\|_{H^\gamma}+ \big\|V^n-v(t_n)\big\|_{H^\gamma}^7\Big),
 \end{align*}
  and
 \begin{align*}
 \big\|\tilde G^n_2(V^n)-\tilde G^n_2(v(t_n))\big\|_{H^\gamma}
     \leq C\tau^2\Big(\big\|V^n-v(t_n)\big\|_{H^\gamma}+ \big\|V^n-v(t_n)\big\|_{H^\gamma}^{13}\Big).
 \end{align*}
From \eqref{est:ln}, we have for $\gamma>\frac32$
 \begin{align}
\big\|\Phi^n(v(t_n))-v(t_{n+1})\big\|_{H^\gamma}\leq C\tau^2.
\end{align}
Furthermore,
from Lemma \ref{def:lemG1} and
 Lemma \ref{def:lemG2}, we find
$$
\big\|\tilde G^n_1(v(t_n))\big\|_{H^\gamma}
   \leq C\tau^2, \quad \big\|\tilde G^n_2(v(t_n))\big\|_{H^\gamma}
   \leq C\tau^4.
$$
Putting together with the above estimates, we conclude that for any $\tau\le 1$,
\begin{align*}
\big\| V^{n+1}-v(t_{n+1})\big\|_{H^\gamma}
   \leq &C\tau^2+(1+C\tau) \big\|V^n-v(t_n)\big\|_{H^\gamma}
        +C\tau\big\|V^n-v(t_n)\big\|_{H^\gamma}^{13},
 \end{align*}
 where the constant  $C$ depends only on $\|u\|_{L^\infty((0,T);H^{\gamma})}$.

 By the iteration and Gronwall inequalities, we get
\begin{align*}
\big\|v(t_{n+1})-V^{n+1}\big\|_{H^\gamma}
\le C\tau^2\sum\limits_{j=0}^n(1+C\tau)^j\le C \tau,\quad n=0,1,\ldots,\frac{T}{\tau}-1.
\end{align*}
This implies the first-order convergence and the following a prior estimate:
\begin{equation}\label{bound}
  \big\|V^n\big\|_{H^\gamma}\leq C, \quad n=0,1\ldots,\frac{T}{\tau}.
\end{equation}
Here the positive constant $C$ depends only on $T$
and $\|u\|_{L^\infty((0,T);H^{\gamma})}$.

In addition, if $u_0\in H^1$,
from \eqref{est:err2} we have
\begin{align}
\big\|\Phi^n(v(t_n))-v(t_{n+1})\big\|_{H^1}\leq C\tau^{\frac32-}.
\end{align}
Furthermore,
from Lemma \ref{def:lemG1} and
 Lemma \ref{def:lemG2}, we find
$$
\big\|\tilde G^n_1(v(t_n))\big\|_{H^1}
   \leq C\tau^{\frac32-}, \quad \big\|\tilde G^n_2(v(t_n))\big\|_{H^1}
   \leq C\tau^{3-}.
$$
Putting together with the above estimates, we conclude that for any $\tau\le 1$,
\begin{align*}
\big\| V^{n+1}-v(t_{n+1})\big\|_{H^1}
   \leq &C\tau^{\frac32-}+(1+C\tau) \big\|V^n-v(t_n)\big\|_{H^1}
        +C\tau\big\|V^n-v(t_n)\big\|_{H^1}^{13},
 \end{align*}
 where the constant  $C$ depends only on $\|u\|_{L^\infty((0,T);H^1)}$.

 By the iteration and Gronwall inequalities, we get
\begin{align*}
\big\|v(t_{n+1})-V^{n+1}\big\|_{H^1}
\le C\tau^{\frac32-}\sum\limits_{j=0}^n(1+C\tau)^j\le C \tau^{\frac12-},\quad n=0,1,\ldots,\frac{T}{\tau}-1.
\end{align*}

Next we prove the almost mass conservation law.
From \eqref{V}, we have
\begin{subequations}
\begin{align}
M(V^{n+1})
     =&M(V^n)\notag\\
     &+2  \big\langle \tilde F^n(V^n), V^n\big\rangle
          +2 \big\langle \tilde G_1^n(V^n), V^n\big\rangle
           +\big\|\tilde F^n(V^n)\big\|_{L^2}^2\label{formula-a}\\
         & +2 \big\langle \tilde G_2^n(V^n), V^n\big\rangle
       +2 \big\langle \tilde F^n(V^n), \tilde G_1^n(V^n)\big\rangle
       +\big\|\tilde G_1^n(V^n)\big\|_{L^2}^2\label{formula-b}\\
        &  +2 \big\langle \tilde F^n(V^n), \tilde G_2^n(V^n)\big\rangle
           +2 \big\langle \tilde G_1^n(V^n), \tilde G_2^n(V^n)\big\rangle
           +\big\|\tilde G_2^n(V^n)\big\|_{L^2}^2\label{formula-c}.
\end{align}
\end{subequations}
By Lemma \ref{def:lemG1} and
 Lemma \ref{def:lemG2}, we get that
 \begin{align*}
 \eqref{formula-a}\leq C\tau^2\Big|M(V^n)-M(u_0)\Big|;
 \quad \eqref{formula-b}\leq C\tau^4\Big|M(V^n)-M(u_0)\Big|,
 \end{align*}
 and
  \begin{align*}
 2\left| \big\langle \tilde G_1^n(V^n), \tilde G_2^n(V^n)\big\rangle\right|
           +\big\|\tilde G_2^n(V^n)\big\|_{L^2}^2
  \leq C\tau^6.
\end{align*}
Therefore we deduce that
\begin{align}\label{errV}
M(V^{n+1})-M(V^n)
 \leq & C\tau^6+C\tau^2\Big|M(V^n)-M(u_0)\Big|
+2 \big\langle \tilde F^n(V^n), \tilde G_2^n(V^n)\big\rangle.
\end{align}
By the definitions of $G^n_2(V^n)$
  in \eqref{def:G1}, we get
 \begin{align*}
2\left|  \big\langle \tilde F^n(V^n), \tilde G_2^n(V^n)\big\rangle\right|
 \le \tilde H^n(V^n)^2\left|   \left\langle \tilde F^n(V^n), V^n\right\rangle\right|
   +2M(u_0)^{-1}\big\langle \tilde F^n(V^n), V^n\big\rangle^2\tilde H^n(V^n).
\end{align*}
Hence, by \eqref{est:H1}, \eqref{est:H} and \eqref {bound}, we obtain
 \begin{align*}
2\left| \big\langle \tilde F^n(V^n), \tilde G_2^n(V^n)\big\rangle\right|
\leq C\tau^6.
\end{align*}
Therefore, we conclude that
\begin{align}\label{est:err}
M(V^{n+1})-M(V^n)
 \leq C\tau^6+C\tau^2\Big|M(V^n)-M(u_0)\Big|.
\end{align}
Then by the iteration, we get
 \begin{align}
\left|M(V^n)-M(u_0)\right|\le  C\tau^5.
\end{align}
This finishes the proof of Theorem \ref{result}.

%
\end{proof}

\section{Numerical experience} \label{sec:numerical}

To set the initial data $u_0(x)$ with the desired regularity, we use the following strategy in \cite{lownls}.
 Choose $N=2^{10}$ and discrete the spatial domain $\bT$ with grid points
$x_j=j\frac{2\pi}{N}$ for $j=0,\ldots,N$.
Take a uniformly distributed random vectors $\mathrm{rand}(N,1)\in[0,1]^N$ and define
\begin{equation}\label{non-smooth}
u_0(x):=\frac{|\partial_{x,N}|^{-\gamma}\mathcal{U}^N}
{\||\partial_{x,N}|^{-\gamma}\mathcal{U}^N\|_{L^\infty}},\quad
x\in\bT, \quad \mathcal{U}^N=\mathrm{rand}(N,1)+i\mathrm{rand}(N,1).
\end{equation}
where the pseudo-differential operator $|\partial_{x,N}|^{-\gamma}$ for $\gamma\geq0$ reads: for Fourier modes $l=-N/2,\ldots$, $N/2-1$,
\begin{equation*}
 \left(|\partial_{x,N}|^{-\gamma}\right)
 _l=\left\{\begin{split}
 &|l|^{-\gamma}\quad \mbox{if}\ l\neq0,\\
  &0\qquad\  \mbox{if}\ l=0.
  \end{split}\right.
\end{equation*}
Thus, we get $u_0\in H^\gamma(\bT)$ for any $\gamma\geq0$. Now we take $\tau=10^{-5}$ and obtain  Figure \ref{fig:NLRI}.

\begin{figure}[t!]
$$\begin{array}{cc}
\psfig{figure=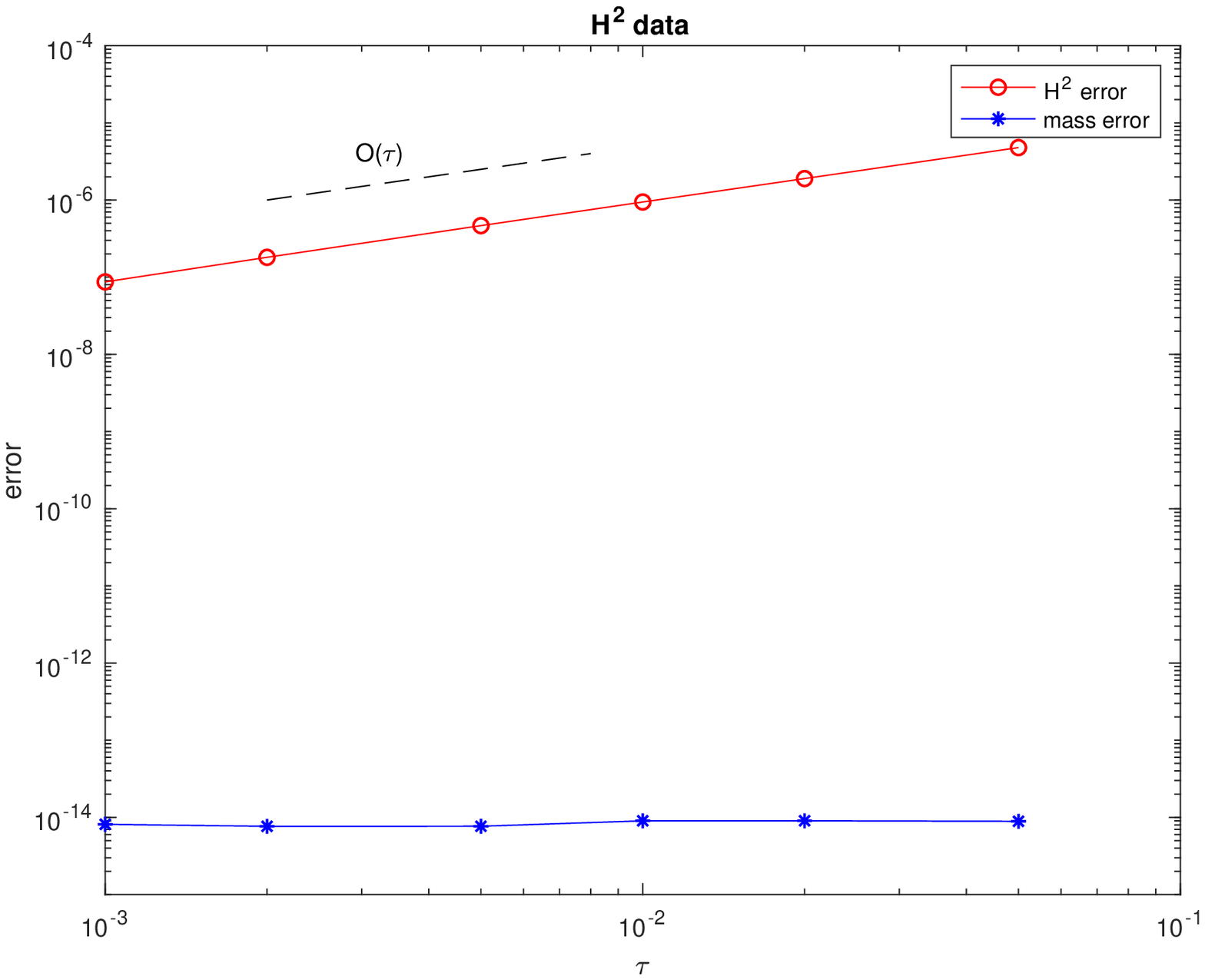,height=6.7cm,width=6.4cm}&
\psfig{figure=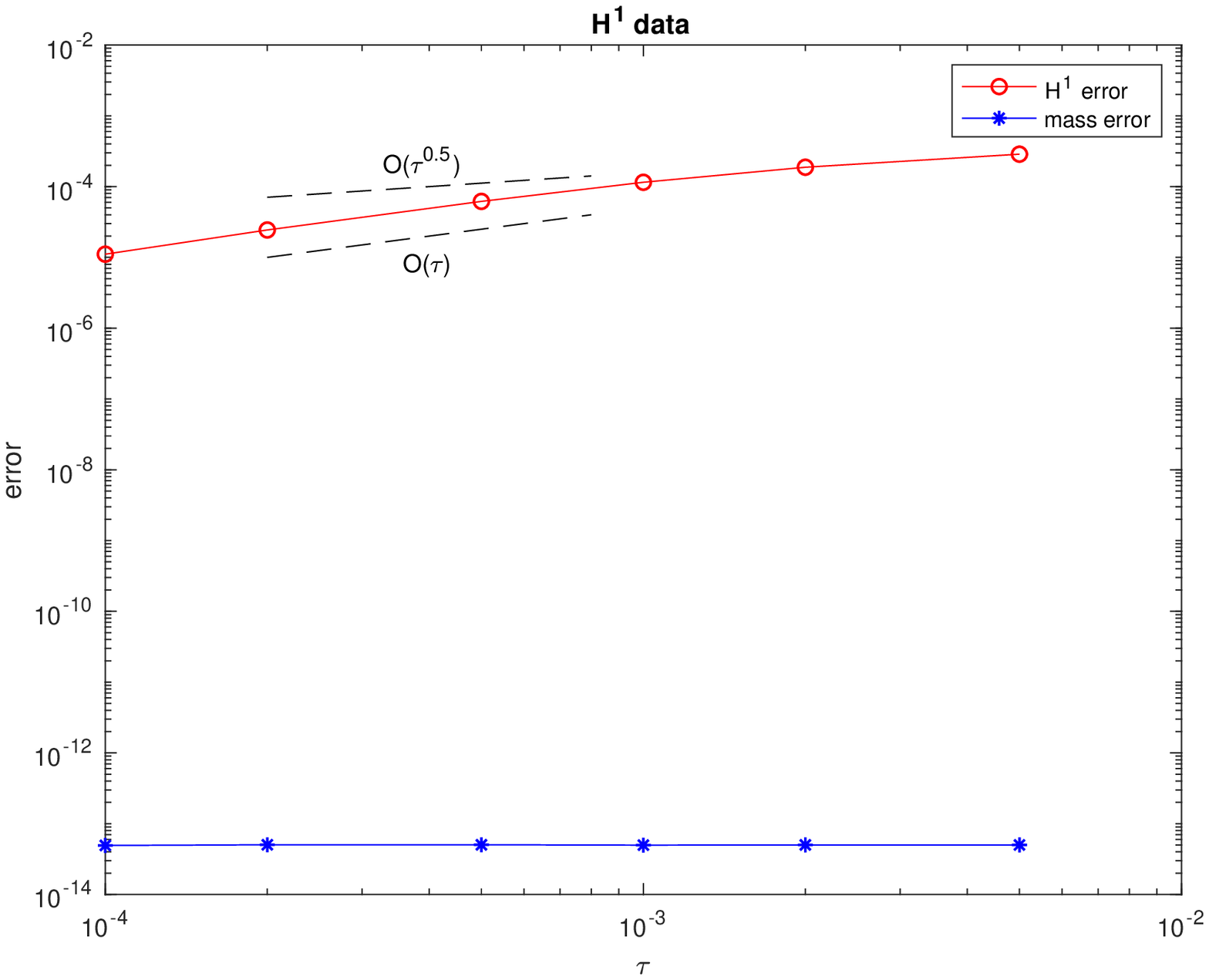,height=6.7cm,width=6.4cm}
\end{array}$$
\caption{Convergence of (NLRI) \eqref{scheme}:   error $\|u-u^n\|_{H^\gamma}$ at $t_n=T=2$  when $\gamma=2$ (left)  and when $\gamma=1$ (right).
}
\label{fig:NLRI}
\end{figure}

The numerical results imply that the scheme \eqref{scheme} has the first-order accuracy of $u(t_n)-u^n$ in $H^2$-norm with initial data in $H^2$, while the mass of the numerical solution is almost conserved (around $10^{-14}$). Furthermore, for $H^1$-data,  $u(t_n)-u^n$ in $H^1$-norm is a little bit better than 0.5 order accuracy.

\section{Conclusion} \label{sec:conclusion}
In this work, we constructed a first-order Fourier integrator for solving the cubic nonlinear Schr\"odinger equation in one dimension. Our designation of the scheme is based on
the exponential-type integration and the Phase-Space analysis of the nonlinear dynamics.
The convergence theorem was established to prove that the first-order accuracy in $H^\gamma$ with initial data in $H^{\gamma}$ for any $\gamma >\frac32$,
where the regularity requirements are lower than existing methods so far.
  Further, we designed a modified numerical scheme
  to obtain the first-order convergence in $H^{\gamma}$ with $H^\gamma$-data meanwhile
  keeps the fifth-order mass convergence.
 By our method, the scheme can be constructed to obtain the arbitrary high-order mass convergence.
 Numerical results were reported to justify the theoretical results.

\vskip 25pt

\bibliographystyle{model1-num-names}

\begin{thebibliography}{00}

\bibitem{besse}
{\sc C. Besse, B. Bid\'{e}garay, and S. Descombes},
Order estimates in time of splitting methods
for the nonlinear Schr\"odinger equation, SIAM J. Numer. Anal. 40 (2002), pp. 26--40.

\bibitem{Bo}
{\sc J. Bourgain}, Fourier transform restriction phenomena for certain lattice subsets and
applications to nonlinear evolution equations. I. Schr\"odinger  equations. Geom. Funct. Anal.
 3  (1993), pp. 107--156.

\bibitem{BoLi-KatoPonce}
{\sc J. Bourgain, D. Li}, On an endpoint Kato-Ponce inequality, Differential Integral Equations 27 (2014) pp. 1037--1072.

\bibitem{cano}
{\sc B. Cano and A. Gonz\'{a}lez-Pach\'{o}n},
 Exponential time integration of solitary waves of cubic
Schr\"odinger  equation, Appl. Numer. Math. 91 (2015), pp. 26--45.

\bibitem{celledoni}
{\sc E. Celledoni, D. Cohen, and B. Owren,}
 Symmetric exponential integrators with an application to the cubic Schr\"{o}dinger equation,
  Found. Comput. Math. 8 (2008), pp. 303--317.

\bibitem{certaine}
{\sc J. Certaine}, The solution of ordinary differential equations with large time constants,
In Mathematical Methods for Digital Computers, Wiley,  (1960), pp. 128--132.


  \bibitem{cohen}
 {\sc  D. Cohen and L. Gauckler,}
  One-stage exponential integrators for nonlinear Schr\"{o}dinger
equations over long times, BIT, 52 (2012), pp. 877--903.




\bibitem{courtes}
{\sc C. Court\`es, F. Lagouti\`ere, F. Rousset},
 Error estimates of finite difference schemes for the Korteweg-de Vries equation,
  IMA J. Numer. Anal. 40 (2020) pp. 628--685.


\bibitem{djuardin}
{\sc G. Dujardin,}
 Exponential Runge-Kutta methods for the Schr\"{o}dinger equation,
 Appl. Numer. Math. 59 (2009), pp. 1839--1857.

\bibitem{faou}
{\sc E. Faou,}
 Geometric Numerical Integration and Schr\"odinger  Equations,
  European Mathematical Society Publishing House, Z\"{u}rich, 2012.


\bibitem{gubinelli}
{\sc M. Gubinelli},
Rough solutions for the periodic Korteweg-de Vries equation,
Comm. Pure Appl. Anal. 11 (2012) pp. 709--733.

  \bibitem{hairer}
{\sc E. Hairer, C. Lubich, and G. Wanner},
Geometric Numerical Integration.
 Structure-Preserving Algorithms for Ordinary Differential Equations,
  2nd ed., Springer, Berlin, 2006.

\bibitem{hersch}
{\sc J. Hersch}, Contribution \`{a} la m\'{e}thode des \'{e}quations aux diff\'{e}rences, Z. Angew. Math. Phys. 9 (1958), pp. 129--180.

\bibitem{hochbruck-l}
{\sc M. Hochbruck, C. Lubich and H. Selhofer},
Exponential integrators for large systems of differential equations,
SIAM J. Sci. Comput. 19 (1998), pp. 1552--1574.


\bibitem{hochbruck2005a}
{\sc M. Hochbruck and A. Ostermann},
 Explicit exponential Runge--Kutta methods for semilinear parabolic problems,
  SIAM J. Numer. Anal. 43 (2005a), pp. 1069--1090.

\bibitem{hochbruck2005b}
{\sc M. Hochbruck and A. Ostermann},
Exponential Runge--Kutta methods for parabolic problems, Appl. Numer. Math. 53 (2005b), pp. 323--339.


\bibitem{Hochbruck}{\sc M. Hochbruck, A. Ostermann}, Exponential integrators, Acta Numer. 19 (2010) 209--286.


 \bibitem{hofmanova}
 {\sc M. Hofmanov\'a, K. Schratz},
  An exponential-type integrator for the KdV equation, Numer. Math.
  136 (2017) pp. 1117--1137.



\bibitem{holden}
{\sc H. Holden, K. H. Karlsen, K.-A. Lie, and N. H. Risebro},
 Splitting for Partial Differential Equations with Rough Solutions,
 European Mathematical Society Publishing House,
Z\"{u}rich, 2010.

\bibitem{holden2011}
{\sc H. Holden, K.H. Karlsen, N.H. Risebro, T. Tao},
 Operator splitting for the KdV equation,
  Math. Comp. 80 (2011) pp. 821--846.

 \bibitem{holden2012}
 {\sc H. Holden, C. Lubich, N.H. Risebro},
  Operator splitting for partial differential equations with Burgers nonlinearity,
  Math. Comp. 82 (2012) pp. 173--185.


  \bibitem{ignat}
{\sc L. I. Ignat,}
A splitting method for the nonlinear Schr\"odinger  equation,
J. Differential Equations, 250 (2011), pp. 3022--3046.

\bibitem{jahnke}
{\sc T. Jahnke and C. Lubich,}
 Error bounds for exponential operator splittings, BIT, 40 (2000), pp. 735--744.

\bibitem{Kato-Ponce}
{\sc T. Kato, G. Ponce}, Commutator estimates and the Euler and Navier-Stokes equations, Commun. Pure Appl. Math. 41 (1988) pp. 891-907.

\bibitem{lownls2}
{\sc M. Kn\"{o}ller, A. Ostermann, K. Schratz},
A Fourier integrator for the cubic nonlinear Schr\"{o}dinger equation with rough initial data,
SIAM J. Numer. Anal. 57 (2019) pp. 1967--1986.




\bibitem{Li-KatoPonce}
{\sc D. Li}, On Kato-Ponce and fractional Leibniz, Rev. Mat. Iberoam. 35 (2019) pp. 23--100.


\bibitem{Lubich}
{\sc Ch. Lubich}, On splitting methods for Schr\"{o}dinger-Poisson and cubic nonlinear Schr\"{o}dinger equations, Math. Comp. 77 (2008) pp. 2141--2153.

\bibitem{Splitting}
{\sc R.I. McLachlan, G.R.W. Quispel}, Splitting methods, Acta Numer. 11 (2002) pp. 341--434.


\bibitem{ostermann}
{\sc A. Ostermann, F. Rousset, K. Schratz},
Error estimates of a Fourier integrator for the cubic Schr\"{o}dinger equation at low regularity,
 arXiv:1902.06779, 2019.

\bibitem{ostermann2020}
{\sc A. Ostermann, F. Rousset, K. Schratz},
Fourier integrator for periodic NLS: low regularity estimates via discrete Bourgain spaces,
 arXiv:2006.12785, 2020.

\bibitem{lownls} {\sc A. Ostermann, K. Schratz},
{Low regularity exponential-type integrators for semilinear Schr\"{o}dinger equations},
 Found. Comput. Math. 18 (2018) pp. 731--755.


\bibitem{ostermann-su}
{\sc A. Ostermann, C. Su},
 A Lawson-type exponential integrator for the Korteweg-de Vries equation,
 to appear on IMA J. Numer. Anal..



\bibitem{pope}
{\sc D. A. Pope}, An exponential method of numerical integration of ordinary differential equations,
 Comm. Assoc. Comput. Mach. 6 (1963), pp. 491--493.




\bibitem{thalhammer}
{\sc M. Thalhammer,}
 Convergence analysis of high-order time-splitting pseudo-spectral methods
for nonlinear Schr\"odinger  equations, SIAM J. Numer. Anal. 50 (2012), pp. 3231--3258.


\bibitem{WuZhao-1} {\sc Y. Wu, X. Zhao},
Optimal convergence of a first order low-regularity integrator for the KdV equation, arXiv:1910.07367, 2019.


\bibitem{wu}
{\sc Y. Wu, X. Zhao},
Embedded exponential-type low-regularity integrators for KdV equation under rough data,
arXiv:2008.07053, 2020.



 \end{thebibliography}

\end{document}